\documentclass[12pt,reqno]{amsart}

\usepackage{times}
\usepackage[T1]{fontenc}
\usepackage{mathrsfs}
\usepackage{latexsym}
\usepackage[dvips]{graphics}
\usepackage{epsfig}
\usepackage{amsmath,amsfonts,amsthm,amssymb,amscd}
\input amssym.def
\input amssym.tex
\usepackage{color}
\usepackage{hyperref}
\usepackage{color}
\usepackage{breakurl}
\usepackage{comment}
\newcommand{\bburl}[1]{\textcolor{blue}{\url{#1}}}

\newcommand\be{\begin{equation}}
\newcommand\ee{\end{equation}}
\newcommand\bea{\begin{eqnarray}}
\newcommand\eea{\end{eqnarray}}
\newcommand\bi{\begin{itemize}}
\newcommand\ei{\end{itemize}}
\newcommand\ben{\begin{enumerate}}
\newcommand\een{\end{enumerate}}
\newcommand\bc{\begin{center}}
\newcommand\ec{\end{center}}
\newcommand\ba{\begin{array}}
\newcommand\ea{\end{array}}

\newtheorem{thm}{Theorem}[section]
\newtheorem{conj}[thm]{Conjecture}
\newtheorem{cor}[thm]{Corollary}
\newtheorem{lem}[thm]{Lemma}

\newtheorem{defi}[thm]{Definition}

\theoremstyle{definition}
\newtheorem{rek}[thm]{Remark}

\numberwithin{equation}{section}

\begin{document}

\title{The Zeckendorf Game}

\author{Paul Baird-Smith}
\email{\textcolor{blue}{\href{mailto:ppb366@cs.utexas.edu}{ppb366@cs.utexas.edu}}}
\address{Department of Computer Science, University of Texas at Austin, Austin, TX}

\author{Alyssa Epstein}
\email{\textcolor{blue}{\href{mailto:alye@stanford.edu}{alye@stanford.edu}}}
\address{Department of Law, Stanford University, Stanford, CA}

\author{Kristen Flint}
\email{\textcolor{blue}{\href{mailto:kflint1101@gmail.com}{kflint1101@gmail.com}}}
\address{Department of Mathematics, Carnegie Mellon University, Pittsburgh, PA 15213}


\author{Steven J. Miller}
\email{\textcolor{blue}{\href{mailto:sjm1@williams.edu}{sjm1@williams.edu}},  \textcolor{blue}{\href{Steven.Miller.MC.96@aya.yale.edu}{Steven.Miller.MC.96@aya.yale.edu}}}
\address{Department of Mathematics and Statistics, Williams College, Williamstown, MA 01267}



\date{\today}


\begin{abstract} Zeckendorf \cite{Ze} proved that every positive integer $n$ can be written uniquely as the sum of non-adjacent Fibonacci numbers. We use this to create a two-player game. Given a fixed integer $n$ and an initial decomposition of $n = n F_1$, the two players alternate by using moves related to the recurrence relation $F_{n+1} = F_n + F_{n-1}$, and whoever moves last wins. The game always terminates in the Zeckendorf decomposition, though depending on the choice of moves the length of the game and the winner can vary. We find upper and lower bounds on the number of moves possible. The upper bound is on the order of $n\log n$, and the lower bound is sharp at $n-Z(n)$ moves, where $Z(n)$ is the number of terms in the Zeckendorf decomposition of $n$. Notably, Player 2 has the winning strategy for all $n > 2$; interestingly, however, the proof is non-constructive.
\end{abstract}

\maketitle

\tableofcontents


\section{Introduction}\label{section}

\subsection{History}

The Fibonacci numbers are one the most interesting and famous sequences. They appear in many varied settings, from Pascal's triangle to mathematical biology. Among their fascinating properties, the Fibonacci numbers lend themselves to a beautiful theorem of Zeckendorf \cite{Ze}: each positive integer $n$ can be written uniquely as the sum of distinct, non-adjacent Fibonacci numbers. This is called the \textit{Zeckendorf decomposition} of $n$ and requires that we define the Fibonacci numbers by $F_1 = 1, F_2 = 2, F_3 = 3, F_4 = 5...$ instead of the usual $1, 1, 2, 3, 5...$ to create uniqueness. The Zeckendorf theorem has been generalized many times (see for example \cite{Ho, Ke, MW1, MW2}), allowing the game explored in this paper potentially to be played similarly on other recurrences. For details on these generalizations, as well as references to the literature on generalizations of Zeckendorf's theorem, see the companion paper \cite{BEFMfibq}.

\subsection{Main Results}

We introduce some notation. By $\{1^n\}$ or $\{{F_1}^n\}$ we mean n copies of $1$, the first Fibonacci number. If we have 3 copies of $F_1$, 2 copies of $F_2$, and 7 copies of $F_4$, we could write either $\{{F_1}^3 \wedge {F_2}^2 \wedge {F_4}^7 \}$ or $\{1^3 \wedge 2^2 \wedge 5^7\}$.

\begin{defi}[The Two Player Zeckendorf Game]\label{defi:zg}
At the beginning of the game, there is an unordered list of $n$ 1's. Let $F_1 = 1, F_2 = 2$, and $F_{i+1} = F_i + F_{i-1}$; therefore the initial list is $\{{F_1}^n\}$. On each turn, a player can do one of the following moves.
\begin{enumerate}
\item If the list contains two consecutive Fibonacci numbers, $F_{i-1}, F_i$, then a player can change these to $F_{i+1}$. We denote this move $\{F_{i-1} \wedge F_i \rightarrow F_{i+1}\}$.
\item If the list has two of the same Fibonacci number, $F_i, F_i$, then
\begin{enumerate}
\item if $i=1$, a player can change $F_1, F_1$ to $F_2$, denoted by $\{F_1 \wedge F_1 \rightarrow F_2\}$,
\item if $i=2$, a player can change $F_2, F_2$ to $F_1, F_3$, denoted by $\{F_2 \wedge F_2 \rightarrow F_1 \wedge F_3\}$, and
\item if $i \geq 3$, a player can change $F_i, F_i$ to $F_{i-2}, F_{i+1}$, denoted by $\{F_i \wedge F_i \rightarrow F_{i-2}\wedge F_{i+1} \}$.
\end{enumerate}
\end{enumerate}
The players alternative moving. The game ends when one player moves to create the Zeckendorf decomposition.
\end{defi}

The moves of the game are derived from the recurrence, either combining terms to make the next in the sequence or splitting terms with multiple copies. We first show the game is well-defined, and then provide bounds on its length.

\begin{thm}\label{thm:welldefined}
Every game terminates within a finite number of moves at the Zeckendorf decomposition.
\end{thm}

Now that we know that the Zeckendorf game is playable, we might wonder how long it will take to play.

\begin{thm}\label{thm:gamelengths}
The shortest game, achieved by a greedy algorithm, arrives at the Zeckendorf decomposition in $n - Z(n)$ moves, where $Z(n)$ is the number of terms in the Zeckendorf decomposition of $n$. The longest game is bounded by $ i*n$, where $i$ is the index of the largest Fibonacci number less than or equal to $n$.
\end{thm}

The theoretical upper bound presented here grows on a log-linear scale because the index of the largest Fibonacci number less than or equal to $n$ is less than $\log_{\phi}(\sqrt{5}F_i + 1/2)$, where $\phi$ is the golden ratio. This relation comes from Binet's formula. Since there is a wide span between the lower bound and the theoretical bound, we simulated random games and were led to the following conjectures.

\begin{conj}\label{conj:gaussian}
As $n$ goes to infinity, the number of moves in a random game decomposing $n$ into it's Zeckendorf expansion, when all legal moves are equally likely, converges to a Gaussian.
\end{conj}

\begin{conj}\label{conj:longgame}
The longest game on any $n$ is achieved by applying splitting moves whenever possible. Specifically, the longest possible game applies moves in the following order: merging ones, splitting from smallest to largest, and adding consecutives, from smallest to largest.
\end{conj}

\begin{conj}\label{conj:avggame}
The average game is of a length linear with $n$.
\end{conj}

Of course, we are interested not just in how long the game takes, but who wins.

\begin{thm}\label{thm:playertwowins}
For all $n>2$, Player 2 has the winning strategy for the Zeckendorf Game.\footnote{If $n=2$, there is only one move, and then the game is over.}
\end{thm}

Since someone must always make the final move, and the game always terminates, for each $n$ one of the two players must have a winning strategy. In other words, someone must always be able to force their victory. This theorem shows that for all nontrivial games, Player 2 has this strategy. The proof is not constructive: it merely shows the existence of Player 2's winning strategy; we cannot identify how they should move. Though we can give exact winning strategies for small $n$, we leave the general winning strategy for future research.


\section{The Zeckendorf Game}

\subsection{The Game is Playable}

In this section, we provide many proofs related to the Zeckendorf Game. We begin with the proof of Theorem \ref{thm:welldefined}, which shows that the game is well defined and playable, starting with an important lemma.

\begin{lem}[Fibonacci Monovariant]\label{lem:fibmon}
The sum of the square roots of the indices on any given turn is a monovariant.\footnote{For us, monovariant is a quantity which is either non-increasing or non-decreasing.}
\end{lem}

\begin{proof}
Our moves cause the following changes in the proposed monovariant. We observe that we only have to consider the affected terms because the suggested monovariant is a sum, so unaffected terms contribute the same before and after the move. Here, $k$ is the index of $F_k$, a term in the current decomposition.\\

\begin{itemize}
\item Adding consecutive terms: $-\sqrt{k-2}-\sqrt{k-1} + \sqrt{k}$\\
\item Splitting: $-2\sqrt{k} + \sqrt{k-2} +\sqrt{k+1}$\\
\item Adding 1's: $-2+\sqrt{2}$\\
\item Splitting 2's: $-2\sqrt{2} + 1 + \sqrt{3}$.\\
\end{itemize}

We note that for all positive $k>2$, in other words all indices not addressed in a special case above, all of these moves cause negative changes. We can see this by the fact that $\sqrt{x}$ is a monotonically increasing, concave function. So this is a monovariant; the sum of the square roots of the indices constantly decrease with each move, so it is strictly decreasing.
\end{proof}

With this lemma, we now prove Theorem \ref{thm:welldefined}.

\begin{proof}[Proof of Theorem \ref{thm:welldefined}]
At the beginning of the game, we have a sum of the square roots of the indices of our list of numbers equal to $\sqrt{n}$, where $n$ is the number we have chosen for the game. From the monovariant of Lemma \ref{lem:fibmon}, we know that the listed moves always decrease this sum. Therefore, no two moves can have the same monovariant value, and there will be no repeat turns. Since the game essentially moves among a subset of partitions of $n$, of which there are a finite number, this implies that the game must always end within a finite number of turns. Moreover, the game always ends at the Zeckendorf decomposition. If it terminated elsewhere, there would either be duplicate terms or the recurrence would apply, by definition. So, there would still be a valid move and the game would not have terminated. This concludes the proof.
\end{proof}

Now that we know for sure that we can play the Zeckendorf Game, we wonder how long the game will take. First, we address the question of whether the game must always take the same amount of turns. If it does, this game is definitely not fair because it predetermines a victor! Fortunately, this is not the case as long as we choose an $n$ greater than $3$.

\begin{lem}\label{lem:multseq}
Given any positive integer $n$ such that $n > 3$, there are at least two distinct sequences of moves $M = \{m_i\}$ where the application of each set of moves to the initial set, denoted $M(\{F_1\}_n)$, leads to $Z_n$, the Zeckendorf decomposition of $n$.
\end{lem}

\begin{proof}
If we show that there are two distinct sets of moves that arrive at the Zeckendorf decomposition of $4$, we have proved the claim because for all $n > 4$: we can follow the two different identified games up to $4$, both of which are valid paths to the Zeckendorf decomposition.

The following two sequences of moves result in the Zeckendorf composition of $4$:
\begin{align*}
M_1 &= \{\{F_1 \wedge F_1 \rightarrow F_2\},\{F_1 \wedge F_1 \rightarrow F_2\}, \{2F_2 \rightarrow F_1\wedge F_3\}\} \\
M_2 &= \{\{F_1 \wedge F_1 \rightarrow F_2\}, \{F_1 \wedge F_2 \rightarrow F_3\}\}
\end{align*}
Therefore, there are multiple games for any $n>3$.
\end{proof}

\begin{rek}\label{rek:nonmult}
If $n\leq 3$ there is one unique sequence of moves that arrives at the Zeckendorf decomposition. If $n = 1$, $M = \{\}$. If $n = 2, M = \{F_1 \wedge F_1 \rightarrow F_2\}$. If $n=3, M = \{\{F_1 \wedge F_1 \rightarrow F_2\},\{F_1 \wedge F_2 \rightarrow F_3\}\}$.
\end{rek}

\begin{cor}\label{cor:multleng}
For any positive $n > 3$, there are at least two games with different numbers of moves. Further, there is always a game with an odd number of moves and one with an even number of moves.
\end{cor}

\begin{proof}
In Lemma \ref{lem:multseq}, we showed that two different sets of moves $M_1$ and $M_2$ arrive at the Zeckendorf Decomposition of $4$. Notice that $|M_1| = 3$ but $|M_2| = 2$. As there are no losing games, for any $n >4$, we can follow either of these games up to the Zeckendorf decomposition of $4$. Regardless of the number or sequence of moves it takes to resolve the rest of the game (call the sequence $M_k$, with $|M_k|=k$), we have already identified two sets of moves with different orders, $M_1 \wedge M_k$ and $M_2 \wedge M_k$, that describe a complete game. $|M_1 \wedge M_k| = 3+k$, but $|M_2 \wedge M_k| = 2+k$. If $k$ is even, $3+k$ is odd and $2+k$ is even. If $k$ is odd, $3+k$ is even and $2+k$ is even. This proves the claim.
\end{proof}

\subsection{Bounds on the Lengths of Games}

We have now established that this game has variation in both game length and parity. It is natural to ask how much variety there is between short, long, and average games. To this end, we provide a proof of Theorem \ref{thm:gamelengths}. To do so, we first include a lemma about the structure of a game following a greedy algorithm.

\begin{lem}\label{lem:zeckmoves}
Let $m(n)$ be the number of moves in a deterministic game where the players must always move on the largest valued number. Let $Z(n)$ be the number of terms in the Zeckendorf decomposition of $n$. Then $m(n) = n - Z(n)$.
\end{lem}

\begin{proof}
Each player acts on the largest valued summand with an available move. The game on $n$ takes $m(n)$ moves. Looking at the game on $n+1$, we observe that the list of summands will eventually reach $ \{1, a,b,c,\dots \} $ where $ \{a,b,c,\dots \} $ is the Zeckendorf decomposition of $n$. Thus $m(n+1) = m(n) + k(n+1)$, where $k$ is a function that is always non-negative.



If the smallest summand in the Zeckendorf decomposition for $n$ is greater than or equal to 3, there are no additional moves that can be made and $k(n+1) = 0$.  However, if the smallest summand is 1 or 2, the smallest summand can be combined with the additional 1.  Because an additional move was completed, $k(n+1)\geq 1$.  It then may be possible to now make another move with the decomposition that was just created.  For every additional move that can be made, $k(n+1)$ increases by 1. We also know that for each additional move, the number of terms in the Zeckendorf decomposition decreases by 1, because each move combines two numbers into one.  We have
\begin{eqnarray}\label{eqn:znplus1}
Z(n+1) & \ = \ & Z(n) + 1 - k(n+1) \nonumber\\
m(n+1) & \ = & \ m(n) + k(n+1).
\end{eqnarray}

Define $t(n)$ by
\begin{equation}
t(n) \ := \ Z(n) + m(n). \\
\end{equation}
By adding the equations given by \eqref{eqn:znplus1} we see that $t(n) $ satisfies a simple recurrence:
\begin{eqnarray}
Z(n+1) + m(n+1) & = & Z(n) +m(n) + 1 \nonumber \\
t(n+1) & = & t(n) + 1  \nonumber \\
& = & t(n-1) + 2 \nonumber \\
& = & t(n-2) + 3 \nonumber \\
& \vdots & \nonumber \\
& = & t(1) + n.
\end{eqnarray}

Since 1 is a Fibonacci number, the Zeckendorf decomposition of 1 is just 1, and we have  $ Z(1) \ = \ 1$ and $ m(1) \ = \ 0 $.  Thus
\begin{eqnarray}
t(n+1) & = & t(1) + n \nonumber \\
& = & Z(1) + m(1) +  n \nonumber \\
& = & 1 + 0 + n \nonumber \\
t(n+1) & = & n+1.
\end{eqnarray}

From this, we see that for any positive integer $n$, $ t(n) \ = \ n $ and so, with the definition of $ t(n)$, we have that
\begin{eqnarray}
t(n) & = & Z(n) + m(n) \nonumber\\
n & = & Z(n) + m(n) \nonumber \\
 m(n) & = & n - Z(n).
\end{eqnarray}

We have shown that for any positive integer $n$, when starting from a list of length $n$ that contains all 1's,  the number of moves it takes to reach the Zeckendorf decomposition for $n$  will be equal to $n$ minus the number of terms in the final Zeckendorf decomposition for $n$. Thus, we have shown Lemma \ref{lem:zeckmoves}.
\end{proof}

\begin{proof}
A quick way to arrive at the Zeckendorf decomposition would be to decrease one term on every move. This would make a short game happen in $n-Z(n)$ moves. No game would be faster, because each possible move decreases the number of terms by at most one. That this game is achievable follows from Lemma \ref{lem:zeckmoves}. Since this number of moves is theoretically shortest and is actually possible, it is a sharp lower bound on the number of moves in the Zeckendorf game.

For the longest game, we return to the monovariant established in Lemma \ref{lem:fibmon}. We observe that the least each move can change the sum is by a splitting move way late into the game. Splitting moves cost at least $2\sqrt{\ell} - \sqrt{\ell-2} -\sqrt{\ell+1}$, where $\ell$ is the index of the largest Fibonacci number less than or equal to $n$. We notice that $2\sqrt{\ell} - \sqrt{\ell-2} -\sqrt{\ell+1} > \sqrt{\ell} - \sqrt{\ell-1}$ because square root is concave and increasing. Then, we observe that $1 = n - (n-1) = (\sqrt{n} - \sqrt{n-1})(\sqrt{n}+\sqrt{n-1})$, which implies that $\sqrt{n} - \sqrt{n-1} = \frac{1}{\sqrt{n}+\sqrt{n-1}} > \frac{1}{n}$. So, $2\sqrt{\ell} - \sqrt{\ell-2} -\sqrt{\ell+1} > 1/\ell$. This gives that it will take at most $\ell \cdot n$ moves to reach the Zeckendorf decomposition. Since $\ell$ is a Fibonacci index, we recall Binet's formula to get a bound in terms of $n$: $F_{\ell}=\frac{1}{\sqrt{5}}({\phi}^{\ell}-(-\phi)^{-\ell})$. We note that $|\frac{\phi^{-\ell}}{\sqrt{5}}|<\frac{1}{2}$, which implies that $\sqrt{5}F_{\ell} < \phi^{\ell} - 1/2$. Taking a base $\phi$ logarithm of both sides, we get $\log_{\phi}(\sqrt{5}F_{\ell} + 1/2) > \ell$. This shows that $\ell\cdot n < \log_{\phi}(\sqrt{5}n + 1/2)n$.
\end{proof}

\subsection{Conjectures on Game Lengths}

Using Mathematica code (see Appendix \ref{sec:code}), we support the conjectures on game length introduced in the introduction with simulation data. We address Conjecture \ref{conj:gaussian} first. Observing Figure \ref{fig: gaussian60}, the best fit Gaussian seems to align well with the distribution of moves taken over 9,999 simulations of the Zeckendorf Game with $n = 60$. Figure \ref{fig:gaussian200} shows the same experiment on $n=200$ with 9,999 simulations.

\begin{figure}[h]
\includegraphics[]{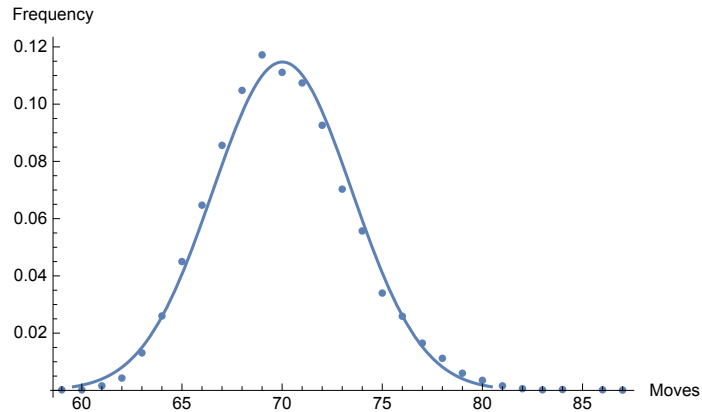}
\caption{Frequency graph of the number of moves in 9,999 simulations of the Zeckendorf Game with random moves when $n=60$ with the best fit Gaussian over the data points.}
\label{fig: gaussian60}
\end{figure}

\begin{figure}[h]
\includegraphics[]{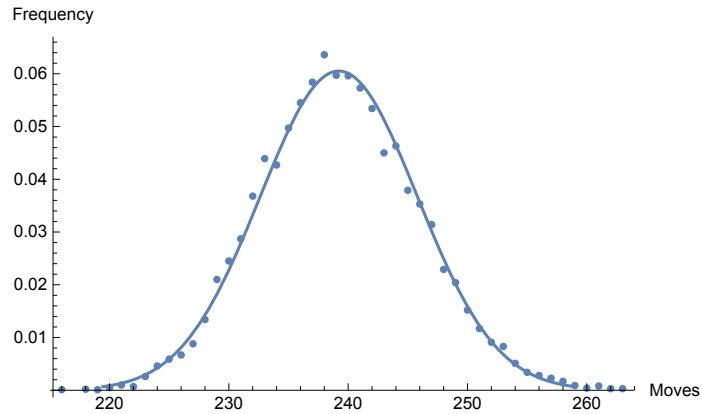}
\caption{Frequency graph of the number of moves in 9,999 simulations of the Zeckendorf Game with random moves when $n=200$ with the best fit Gaussian over the data points.}
\label{fig:gaussian200}
\end{figure}

To see how Conjecture \ref{conj:longgame} may be true, we provide two pieces of evidence. The first is the move count from simulation of the deterministic algorithm stated in the conjecture. Recall that the order of moves is adding ones, splitting from smallest to largest, then adding consecutives from smallest to largest. Figure \ref{fig:longgamedata} shows an array with the $x$ component being $n$ and the $y$ component being the number of moves in the hypothesized deterministic longest game algorithm. The second piece of evidence comes from a Java program, a link to and readme for which is included in Appendix \ref{sec:code}. The Java program explores all possible moves in the Zeckendorf game for a given $n$. The data produced here is the longest possible move length for the $n$ listed. Observe that the two arrays provide identical data. This suggests that the hypothesized longest game algorithm may actually be the theoretically longest game on each $n$.

\begin{figure}[h]
\includegraphics[scale=.7]{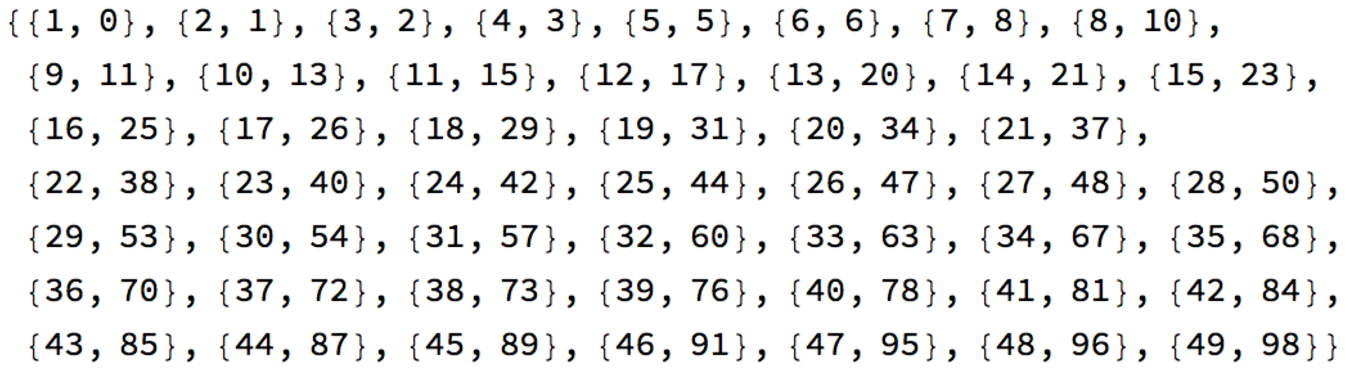}
\caption{Data taken from the simulation of the deterministic longest game proposed by the algorithm in Conjecture \ref{conj:longgame}.}
\label{fig:longgamedata}
\end{figure}

\begin{figure}[h]
\includegraphics[scale=.6]{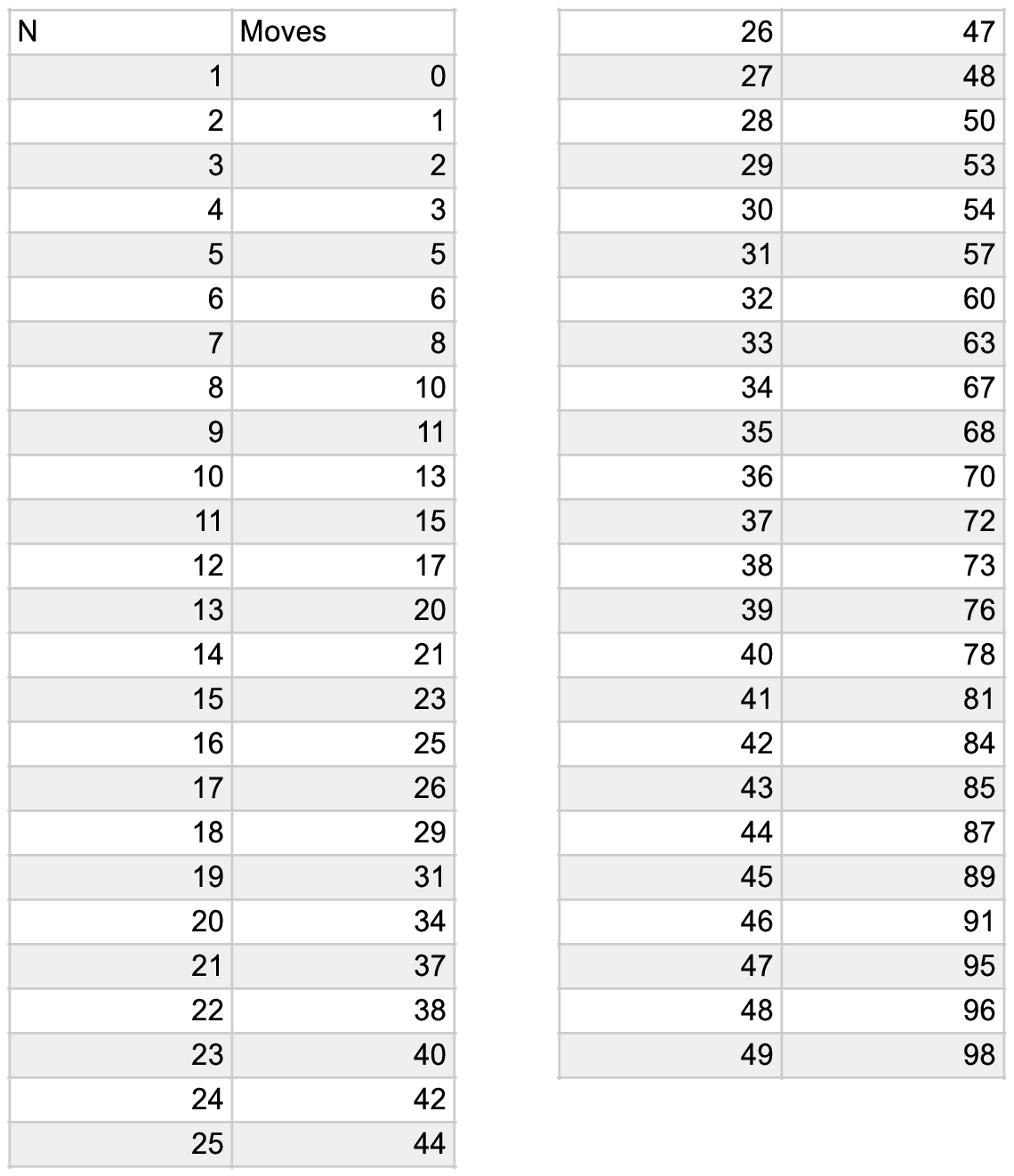}
\caption{Computer proven data of the number of moves in the longest route to victory courtesy of the Java code written by Paul Baird-Smith.}
\end{figure}

In support of Conjecture \ref{conj:avggame}, we offer the graph in Figure \ref{fig:movesavg}. Using data from simulating the Zeckendorf game on varying $n$, we plot the average number of moves in a game against $n$. We observe that a best fit line with slope of around $1.2$ fits the data points well. Due to computer restraints, we are unable to provide data beyond $n = 200$ (not pictured in the graph, but included in the data). The average taken on $n =200$ is $239$, very close to $1.2 \cdot 200$.

\begin{figure}[h]
\includegraphics[]{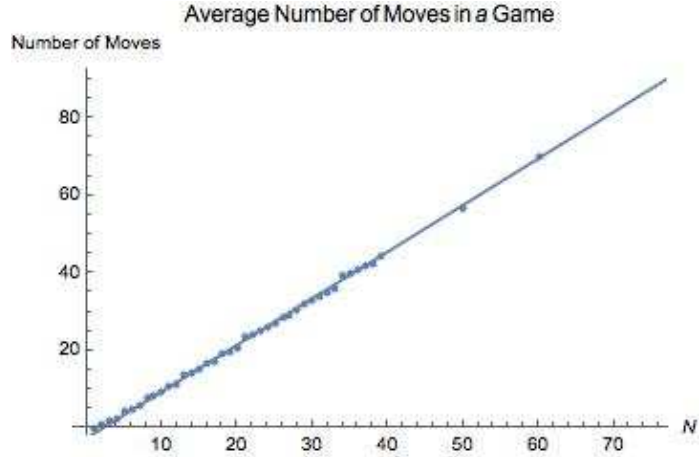}
\caption{Graph of the average number of moves in the Zeckendorf game with simulations ranging from 999 to 9,999 for varying $n$ with the best fit line over the data points.}
\label{fig:movesavg}
\end{figure}

\subsection{Winning Strategies}

Since someone must always make the final move, and the game always ends at the Zeckendorf decomposition, there are no ties. Therefore one player or the other has a winning strategy on each $n$. This section is devoted to the proof that Player 2 has the winning strategy for all $n >2$, the statement of Theorem \ref{thm:playertwowins}. For this proof, we use a visual aid provided in Figure \ref{fig:tree}.

\begin{figure}[h]
\includegraphics[scale=.8]{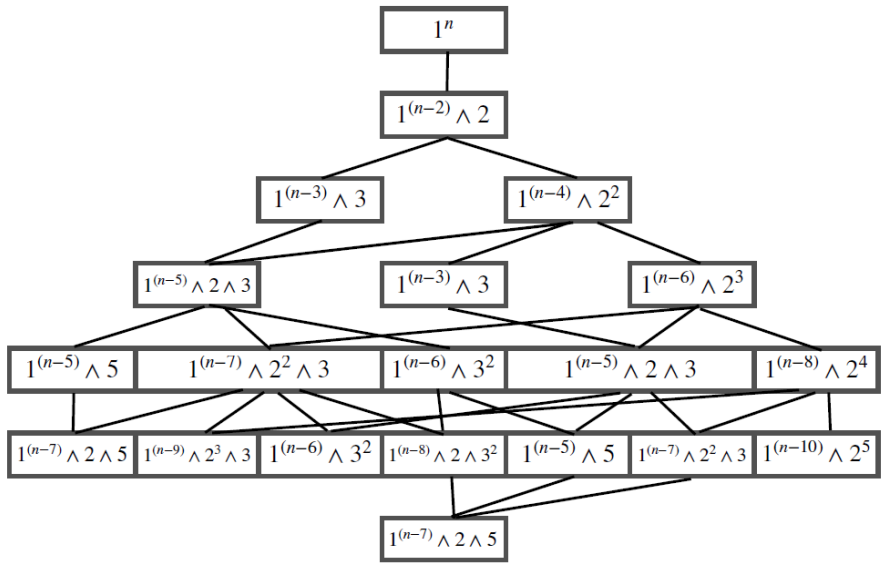}
\caption{Tree depicting the general structure of the first several moves of the Zeckendorf game.}
\label{fig:tree}
\end{figure}

\begin{proof}
Assume that Player 1 wins the game. Therefore, Player 1 must have the winning strategy from the node in the first row with the caption $\{1^{(n)}\}$. We color this node red in Figure \ref{fig:treeproof}. Since this node only has one child, Player 1 must have the winning strategy from $\{1^{(n-2)}\wedge 2\}$ in row two. Player 2 makes the next move, so Player 1 must have the winning strategy from both the nodes in row 3; if not, Player 2 would move to the one from which Player 1 did not have the winning strategy. We focus on the children of the node $\{1^{(n-3)}\wedge 3\}$ in row 3. This node has one descendant only; therefore $\{1^{(n-5)}\wedge 2\wedge 3\}$ in row 4 must have a winning strategy for Player 1. Player 2 makes the move next, so all three children of $\{1^{(n-5)}\wedge 2\wedge 3\}$ in row 5 must be a winning strategy for Player 1. Observe that one such child is $\{1^{(n-5)}\wedge 5\}$ in row 5. If Player 1 has the winning strategy from that node in row 5, if that node is on the next layer, in row 6, following the same winning strategy, Player 2 can win from the row 6 node $\{1^{(n-5)}\wedge 5\}$. So we color that node blue on row 6 of Figure \ref{fig:treeproof} to indicate Player 2 having a winning strategy. Since that node has only one child, $\{1^{(n-7)}\wedge 2 \wedge 5 \}$ in row 7, Player 2 must have a winning strategy from that node. This means that any parent of this node must be a winning strategy location for Player 2 because Player 2 could just move to $\{1^{(n-7)}\wedge 2 \wedge 5 \}$ in row 7 from those parents. This means that $\{1^{(n-8)}\wedge 2 \wedge 3^{(2)} \}$ in row 6 must have a winning strategy for Player 2; however, since both children in row 6 of $\{1^{(n-6)}\wedge 3^{(2)}\}$ in row 5 have winning strategies for Player 2, this means the row 5 node must be a winning strategy for Player 2, not Player 1 as we had earlier deduced. This leads to a contradiction that proves the claim for $n$ sufficiently large ($n \geq 9$). For the small cases of $2<n<9$, computer code such as the one referenced in Appendix \ref{sec:code} can show that Player 2 has the winning strategy by brute force.
\end{proof}

\begin{figure}[h]
\includegraphics[scale=.8]{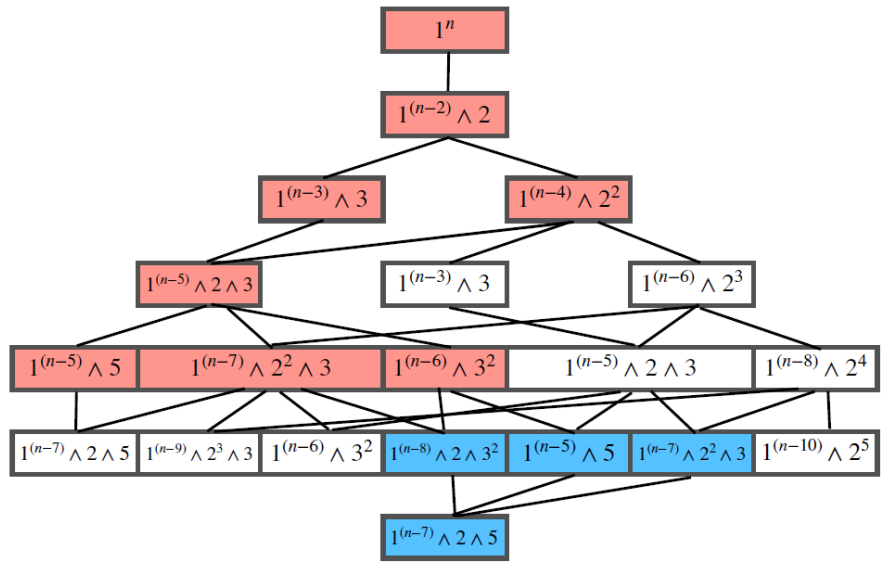}
\caption{Tree depicting the proof of Theorem \ref{thm:playertwowins}. Red boxes have a winning strategy for Player 1, and blue boxes indicate a winning strategy for Player 2.}
\label{fig:treeproof}
\end{figure}

This result is non-trivial and surprising. Game trees for large $n$ have many, many nodes, with no obvious path to victory for either player (see Figure \ref{fig:smalltree} for $n=9$ and Figure \ref{fig:bigtree} for $n=14$ for an example of how quickly the number of nodes grows). Additionally, this is merely an existence proof, which means we cannot tell how Player 2 should move to achieve his victory. This makes the game less rigged for human players; indeed, random simulations of the games show Players 1 and 2 winning roughly even amounts of the time.

\begin{figure}[h!]
\includegraphics[scale=.8]{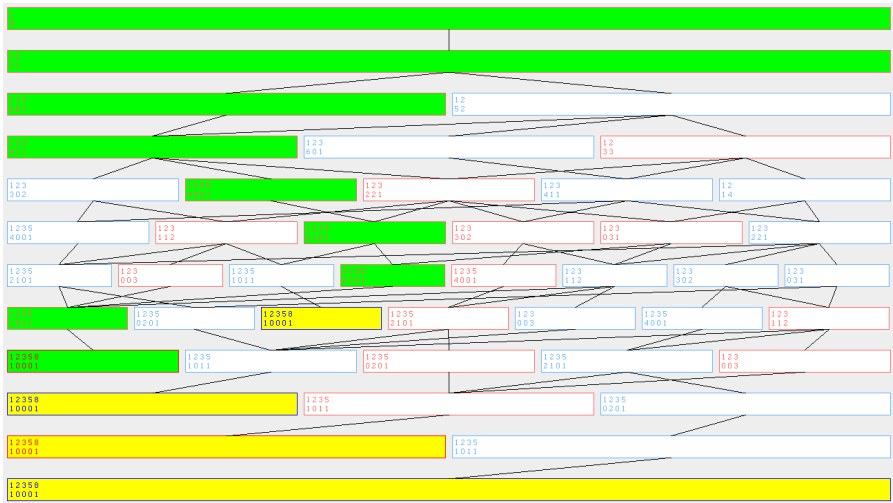}
\caption{Game tree for $n=9$, showing a winning path in green. Image courtesy of the code referenced in Appendix \ref{sec:code}.}
\label{fig:smalltree}
\end{figure}

\begin{figure}[h!]
\includegraphics[scale=.8]{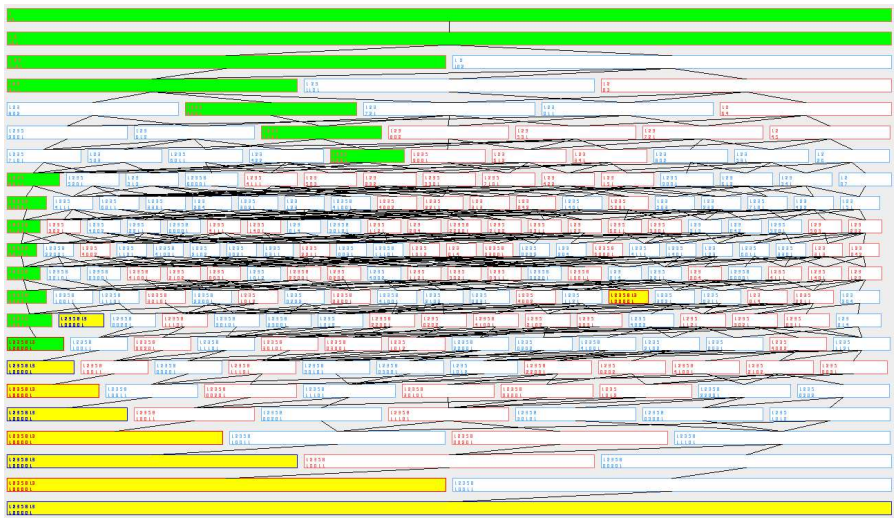}
\caption{Game tree for $n=14$, showing a winning path in green. Image courtesy of the code in Appendix \ref{sec:code}.}
\label{fig:bigtree}
\end{figure}


\section{Future Work}

There are many more ways that studies of this game can be extended. This paper covered the Zeckendorf Game quite extensively, but improved upper bounds may still be found on the number of moves in any game. This work also showed the existence of a winning strategy for player two for all $n>2$, but it does not show what that strategy is.

The Zeckendorf Game is on the Fibonacci recurrence; however, the fact that Zeckendorf's theorem generalizes means that the game could be played on other recurrences. Finding which classes of recurrences have meaningful games, bounding the moves on those games, and considering winning strategies are all fruitful avenues for further exploration.

Expanding in another direction, the Zeckendorf Game as conceived of by this thesis is a two-player game. What if more players want to join? Who wins in that case, for either the Generalized or regular Zeckendorf Game? The analysis done here only shows there is a winning strategy that takes an even number of moves for all $n>2$ for the Zeckendorf Game. It says nothing about the number of moves modulo $k$, where $k$ is odd and greater than 2!

\appendix


\section{Code}\label{sec:code}

Programs for  simulating a random version of the Zeckendorf game, running a deterministic worst game algorithm of the Zeckendorf game, and simulating a random Tribonacci Zeckendorf game is available at \begin{center}\bburl{github.com/paulbsmith1996/ZeckendorfGame/blob/master/ZeckGameMathematica.nb}. \end{center}.

TreeDrawer is used to give a visual representation of the tree structure of the
Zeckendorf game. It plays through a specified game,
determining all moves that can be made, and draw all possible paths to the end of
this game. The ReadMe file can be found at \bburl{https://github.com/paulbsmith1996/ZeckendorfGame}.
TreeDrawer can be executed, after compilation, by running the command
\begin{center}
appletviewer TreeDrawer.java
\end{center}
Do not delete the comment in the preamble, as this is used at runtime by the
appletviewer. Email paul.bairdsmith@gmail.com for more information.


\ \\

\end{document}